\documentclass[12pt]{amsart}

\usepackage{amssymb,latexsym, amssymb, amsthm}
\usepackage{enumerate}
\usepackage{amsfonts}
\usepackage{color}
\usepackage{comment}
\usepackage{graphicx}

\usepackage{hyperref}

\newtheorem{theorem}{Theorem}[section]
\newtheorem{lemma}[theorem]{Lemma}

\newcommand{\N}{\mathbb{N}}
\newcommand{\R}{\mathbb{R}}
\newcommand{\Z}{\mathbb{Z}}

\begin{document}

\title[$k$-free numbers]{Moment estimates for exponential sums\\ over $k$-free numbers}

\author[Eugen Keil]{Eugen Keil}

\address{Department of Mathematics \\ Bristol University\\
University Walk, Clifton, Bristol BS8 1TW, United Kingdom}

\email{maxek@bristol.ac.uk}

\date{\today}


\maketitle

\begin{abstract}
We investigate the size of $L^p$-integrals for exponential sums over $k$-free numbers
and prove essentially tight bounds. 
\end{abstract}

\section{Introduction} \label{Intro}

Squarefree numbers are known to have a structure with complexity somewhere 
between periodic functions and the primes. 
We investigate $L^p$-integrals of exponential sums of $k$-free numbers
and show how their behaviour reflects this intuition in a quantitative way.
Define the indicator function of $k$-free numbers $\mu_k$ for $k \geq 2$ by
\begin{align*}
\mu_k(n) = \left\{\begin{array}{cc} 1 &  \mbox{if there is no } d > 1 
\mbox{ such that } d^k|n, \\ 0 & \mbox{otherwise.}  \end{array} \right.
\end{align*}
For the parameter $N \in \N$ the corresponding exponential sum is given by
\begin{align*}
S_k(\alpha) = \sum_{n \leq N} \mu_k(n) e(\alpha n),
\end{align*}
where $e(x) = \exp(2\pi i x)$ as usual.
Br\"udern, Granville, Perelli, Vaughan and Wooley \cite{BGPVW} 
gave (among other interesting results) the first non-trivial bounds on the $L^1$-mean for this exponential sum.
Balog and Ruzsa \cite{BR} improved upon these estimates and gained essentially best possible estimates for
the $L^1$-mean in their attempt to understand lower bounds for $L^1$-means of the M\"obius function.
We state their result which is useful for us later on.

\begin{theorem}\label{BR-Thm1}
We have
\begin{align*}
N^{1/(k+1)} \ll \int_0^1 |S_k(\alpha)| \, d\alpha \ll N^{1/(k+1)}.
\end{align*}
\end{theorem}
The Vinogradov notation $f \ll g$ is used to express that $|f(x)| \leq C g(x)$ for some constant $C>0$. 
Equivalently, in $O$-notation, we write $f = O(g)$.

While the $L^1$-case is interesting in itself, it is possible to gain a deeper understanding
by inspecting the $L^p$-norms for various $p \geq 0$, as we will see shortly.
We combine interpolation techniques with some methods from \cite{BR} to obtain
an almost complete understanding of the order of magnitude for the $L^p$-means.
Define the $p$-th moment integral for $p \geq 0$ by
\begin{align*}
I_k(p) := \int_0^1 |S_k(\alpha)|^p \,d\alpha.
\end{align*}

The main results of this paper are the following estimates.

\begin{theorem} \label{main-Thm}
For $k \geq 2$ and $N \geq 2$, we have
\begin{align*}
N^{p-1} \ll\ & I_k(p) \ll N^{p-1}  &\mbox{\quad for \quad} p > 1+1/k,\\
 N^{1/k}\log N \ll\ & I_k(p) \ll N^{1/k}(\log N)^2 &\mbox{\quad for \quad} p = 1+1/k,\\
N^{p/(k+1)} \ll\ & I_k(p) \ll N^{p/(k+1)} &\mbox{\quad for \quad} p < 1+1/k,
\end{align*}
where the implied constants may depend on $k$ and $p$.
\end{theorem}
 
It is possible to get a more intuitive feeling for this statement by
plotting $E(p) = \inf\{r \in \R: I_k(p) \ll N^r\}$.
For $k = 2$ it gives the following picture.

\begin{center}
\includegraphics[width=50mm]{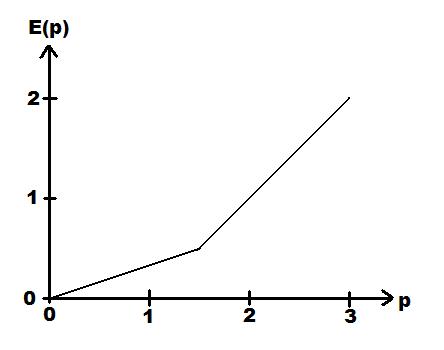}
\end{center}

One interpretation of this result is to see the critical point $p = 1 + 1/k$, 
where the asymptotic growth behaviour changes, as a measure of regularity.
The corresponding graph for periodic functions, for example, has its critical point 
at $p = 1$. 
From the work of Vaughan \cite{Vau} on the $L^1$-norm
for the exponential sum of the von Mangoldt-function, we obtain
the critical point at $p = 2$, as can be seen by the methods of Section \ref{LB}.
The intermediate behaviour of $k$-free numbers in this aspect wasn't observed in the
literature so far and adds a quantitative aspect to
the intuitive feeling of regularity of $k$-free numbers.

On the practical side, the theorem provides the sharp bound $O(N^{p-1})$ for values of $p$ less than two. 
This saves many technical steps in applications of the circle method to additive problems involving $k$-free numbers.
The investigation of additive problems involving $k$-free numbers was started in a series 
of papers of Evelyn and Linfoot \cite{EL}.
Their results were refined by Mirksy \cite{M} and recent work on this topic can be found in \cite{BP}, for example.
As another immediate application we mention that the pointwise bounds for 
$S_k(\alpha)$ in \cite{SP} and \cite{T}
can be combined with our $L^p$ results to improve on the minor-arc estimate
in Theorem 1.3 of \cite{BGPVW}.

The author believes that the lower bound for $p = 1 + 1/k$ corresponds to the true order of
magnitude and it would be nice to have a tight upper estimate as well.

The methods of this paper extend to a wider range of functions coming from a convergent sieve process.
If $S_f(\alpha) = \sum_{n \leq N} f(n) e(\alpha n)$ is the exponential sum for the function $f$,
it is natural to ask for conditions under which one can obtain bounds of the form 
$\int_0^1 |S_f(\alpha)|^p \, d\alpha \ll N^{p-1}$ for some $p < 2$.
Estimates of this type are of interest for applications of the circle method.
Similar behaviour should be observable in the multidimensional setting as well. The generalized $\gcd$-functions 
in Section 6 of \cite{EK} might be a natural example to look at.

\section{Preparation and Lemmata} \label{Prep}

We treat $k$ as constant throughout the paper and for notational simplicity often don't indicate 
the dependence of functions on $k$. 
All constants hidden in the Vinogradov and $O$-notation may also depend on $k$.
Write $[x]$ for $\max \{z \in \Z: z \leq x\}$, the integer part of $x \in \R$,
and $\|x\|$ for $\min \{|z-x|: z \in \Z\}$, the distance of $x$ to the nearest integer.
Denote the greatest common divisor of $a$ and $b$ by $(a;b)$ and write lcm$(a,b)$
for the least common multiple.  
The function $\mu_k$ has a well-known decomposition of the form
\begin{align} \label{mu-identity}
\mu_k(n) = \sum_{d^k|n} \mu(d).
\end{align}
Based on this decomposition, we define the functions $b_{y,z}(n)$ by 
putting $b_{y,z}(0) = 0$, and when $n \neq 0$ by taking 
\begin{align} \label{def-byz}
b_{y,z}(n) = \sum_{\substack{d^k|n\\y \leq d < z}} \mu(d).
\end{align}
We give a simple but crucial lemma from \cite{BR} rewritten in our notation.

\begin{lemma} \label{BR-Lemma1}
For any $1 \leq K \leq N$ and $1 \leq y < z$ we have
\begin{align*}
\sum_{N-K < n \leq N} |b_{y,z}(n)|^2 \ll Ky^{1-k} + N^{1/k}\log^3(z).
\end{align*}
The second summand on the right hand side is not necessary if $K = N$.
\end{lemma}

\begin{proof} We follow the argument of \cite{BR}.
Insert the definition of $b_{y,z}(n)$ and open the square, to arrive at
\begin{align} \label{eq-NK2}
\sum_{N-K < n \leq N} |b_{y,z}(n)|^2 & = 
\sum_{y \leq d,e < z} \mu(d)\mu(e) \sum_{\substack{N-K < n \leq N \\d^k|n,\, e^k|n}} 1.
\end{align}
The inner sum is bounded by $K/\mbox{lcm}(d^k,e^k) + 1$ if $\mbox{lcm}(d,e) \leq N^{1/k}$
and is empty otherwise.
Therefore, when it is non-empty, the right hand side of \eqref{eq-NK2} bounded by
\begin{align*}
 \sum_{y \leq d,e < z} \Big(\frac{K}{\mbox{lcm}(d^k,e^k)} + \frac{N^{1/k}}{\mbox{lcm}(d,e)} \Big)
 \leq \sum_{y \leq d,e} \frac{K}{d^ke^k} \sum_{t|(d;e)} t^k  + \sum_{d,e < z} \frac{N^{1/k}}{de} \sum_{t|(d;e)} t.
\end{align*}
The last inequality used the identity $(d;e) \cdot \mbox{lcm}(d,e)  = de$ with the trivial observation, that a sum
of positive terms is always bigger then one of its summands. Changing the order of summation in the above sums,
we finally obtain the upper bound
\begin{align*}
 \sum_{N-K < n \leq N} |b_{y,z}(n)|^2 \leq &\ K \sum_{t} t^k 
\Big( \sum_{y \leq d,\, t|d} d^{-k} \Big)^2  + N^{1/k} \sum_{t < z} t \Big(\sum_{d < z,\, t|d} d^{-1} \Big)^2\\
\leq &\ K \sum_{t} t^{-k} \min(1,(t/y)^{k-1})^2  + N^{1/k} \sum_{t < z} t^{-1} \Big(\sum_{d < z} d^{-1} \Big)^2\\
\ll &\ K \sum_{t \leq y} t^{k-2}y^{2-2k}  + K \sum_{t > y} t^{-k} +  N^{1/k} \log^3 z \\
\ll &\ Ky^{1-k} + N^{1/k} \log^3 z.
\end{align*}
In the case where $K = N$ we get the bound $K/\mbox{lcm}(d^k,e^k)$ for the inner sum in \eqref{eq-NK2}
instead of $K/\mbox{lcm}(d^k,e^k) + 1$, which removes the term $N^{1/k} \log^3 z$ from the final estimate. 
\end{proof}

Let $H$ and $h$ be integers satisfying the inequalities $N^{1/k} \leq 2^H < 2N^{1/k}$ and 
$N^{1/(k+1)} < 2^h \leq 2N^{1/(k+1)}$. Define by means of \eqref{def-byz}
\begin{align*}
b_{i}(n) = b_{2^{i-1},2^i}(n) \qquad \mbox{  and  } \qquad
b^*(n) = b_{2^h,2^H}(n).
\end{align*}
Then from \eqref{mu-identity}, we obtain two decompositions
\begin{align} \label{mu-decomp}
\mu_k(n) = \sum_{i \leq H} b_i(n) = \sum_{i \leq h} b_i(n) + b^*(n)
\end{align}
and denote the corresponding exponential sums by
\begin{align*}
T_i(\alpha) = \sum_{n \leq N} b_i(n) e(\alpha n) \mbox{\quad and \quad} T^*(\alpha) = \sum_{n \leq N} b^*(n) e(\alpha n).
\end{align*}
The main idea of this work can be summarized as follows.
The functions $b_i$ with small $i$ are very regular periodic functions and the $T_i(\alpha)$ have a small $L^1$-norm, 
while the functions with bigger $i$ have low density and, therefore, have a small $L^2$-norm by Parseval's identity, 
as can already be seen in Lemma \ref{BR-Lemma1} above.
If we interpolate the corresponding exponential sums between the $L^1$- and $L^2$-estimates by H\"older's inequality,
we save enough to give the desired result up to some logarithmic factors. 
Some more effort is needed to remove those factors as well.

The next lemma summarizes the necessary $L^1$- and $L^2$-estimates for the exponential sums
together with a crude $L^p$-bound for $T_i$ which is useful in the removal of the logarithmic factors later on.

\begin{lemma} \label{L1L2-est}
For $N \geq 2$ and $p > 1$, we have
\begin{align*} 
& \int_{0}^1 |T_i(\alpha)| \,d\alpha \ll 2^i \log N, \qquad 
\int_{0}^1 |T_i(\alpha)|^p \,d\alpha \ll_p 2^i N^{p-1},\\
& \int_{0}^1 |T_i(\alpha)|^2 \,d\alpha \ll 2^{-i(k-1)} N , \quad 
\int_{0}^1 |T^*(\alpha)|^2 \,d\alpha \ll N^{2/(k+1)},
\end{align*}
where the implicit constants may depend on $k$.
\end{lemma}

\begin{proof}
We prove the first two estimates at once. Start with
\begin{align*} 
\int_{0}^1 |T_i(\alpha)|^p \,d\alpha 
= & \int_{0}^1 \Big|\sum_{n \leq N} \sum_{\substack{2^{i-1} \leq d < 2^i\\ d^k|n}} \mu(d) e(\alpha n) \Big|^p \,d\alpha\\
\leq & \int_{0}^1  \Big(\sum_{2^{i-1} \leq d < 2^i} \Big| \sum_{m \leq N/d^k} e(\alpha d^km) \Big|\ \Big)^p \,d\alpha.
\end{align*}
For $p > 1$ and $p' = p/(p-1)$ we apply H\"older's inequality inside and get 
\begin{align*} 
\int_{0}^1 |T_i(\alpha)|^p \,d\alpha \leq &\ 
\int_{0}^1 \Big(\sum_{2^{i-1} \leq d < 2^i} 1^{p'}\Big)^{p/p'} 
\sum_{2^{i-1} \leq d < 2^i}  \Big|\sum_{m \leq N/d^k} e(\alpha d^km)\Big|^p \,d\alpha\\
\leq &\ 2^{i(p-1)} \sum_{2^{i-1} \leq d < 2^i}  \int_{0}^1 \Big|\sum_{m \leq N/d^k} e(\alpha d^km)\Big|^p \,d\alpha.
\end{align*}
For $p = 1$ we just change the order of integration and summation instead. 
Now we perform the change of variables $\beta = d^k\alpha$ and use 1-periodicity of the exponential $e(\beta m)$
to estimate the integral by
\begin{align*} 
\int_{0}^1 \Big|\sum_{m \leq N/d^k} e(\alpha d^km)\Big|^p \,d\alpha = & \
d^{-k} \int_{0}^{d^k} \Big|\sum_{m \leq N/d^k} e(\beta m)\Big|^p \,d\beta\\
= \int_0^1 \Big|\sum_{m \leq N/d^k} e(\beta m)\Big|^p \,d\beta
\ll &\ \int_0^1 \min(N/d^k,\|\beta\|^{-1})^p \,d\beta.
\end{align*}
Divide the range of integration into $\|\beta\| \leq d^k/N$ and the rest to get
\begin{align*} 
\int_{0}^1 |T_i(\alpha)|^p \,d\alpha \ll 2^{i(p-1)} \sum_{2^{i-1} \leq d < 2^i} \Bigg( 
d^kN^{-1} (Nd^{-k})^p + \int_{d^k/N}^{1/2} |\beta|^{-p} \,d\beta \Bigg).
\end{align*}
Dependent on whether $p = 1$ or $p > 1$ we get the corresponding estimate.\\

The $L^2$-estimates both follow easily from Lemma \ref{BR-Lemma1} and Parseval's identity
as follows. One has
\begin{align*} 
\int_{0}^1 |T_i(\alpha)|^2 \,d\alpha = \sum_{n \leq N} |b_i(n)|^2 \ll N 2^{i(1-k)},
\end{align*}
where we used the sharpened version in the case $K = N$. Similarly
\begin{align*} 
\int_{0}^1 |T^*(\alpha)|^2 \,d\alpha = \sum_{n \leq N} |b^*(n)|^2 \ll N 2^{h(1-k)} \ll N^{2/(k+1)},
\end{align*}
since $2^h \geq N^{1/(k+1)}$.
\end{proof}

Before we conclude this section and move on to the main part of the proof, we state 
a standard reformulation of H\"older's inequality, which the author found easier to work with
in some of the later passages.

\begin{lemma} \label{H-int} (H\"older-interpolation)
Assume that for $1 \leq q_1 \leq q_2$ we have
\begin{align*}
\int_0^1 |f(\alpha)|^{q_1} \,d\alpha \ll X^{a_1} \mbox{ and } \int_0^1 |f(\alpha)|^{q_2} \,d\alpha \ll X^{a_2}.
\end{align*}
Then for any $p = (1-\theta)q_1 + \theta q_2$ with $\theta \in [0,1]$ we have 
\begin{align*}
\int_0^1 |f(\alpha)|^p \,d\alpha \ll X^{(1-\theta)a_1 + \theta a_2}.
\end{align*}
\end{lemma}

\section{Upper Bounds for $p = 1 + 1/k$} \label{Crit}

Now we are well prepared to look at the critical point $p = 1 + 1/k$. This is conceptually the easiest
case and provides an outline for the general procedure, which in the other cases
is hidden by some technical details.

Using the first decomposition of $\mu_k$ in \eqref{mu-decomp} and H\"older's inequality, we can write 
(using $p' = k+1$ and $H \approx k^{-1}\log_2 N$)
\begin{align*}
I_k(p)  = & \int_0^1 \Big|\sum_{i \leq H} T_i(\alpha)\Big|^p \, d \alpha
 \ll_p \int_0^1 \Big(\sum_{i \leq H} 1^{p'} \Big)^{p/p'} \sum_{i \leq H} |T_i(\alpha)|^p \, d \alpha\\
 \ll & (\log N)^{1/k} \sum_{i \leq H} \int_0^1|T_i(\alpha)|^p \, d \alpha.
\end{align*}
By losing a logarithmic factor, we have separated the different pieces and can now interpolate each
$T_i$-integral between their $L^1$- and $L^2$-norms. We can write 
$p = 1+1/k = (1-\theta)\cdot 1 + \theta \cdot 2$ with $\theta = 1/k$ and, therefore,
by applying Lemma \ref{H-int} to interpolate the estimates from Lemma \ref{L1L2-est}, we have
\begin{align*}
\int_0^1|T_i(\alpha)|^p \, d \alpha \ll  (2^i \log N)^{1-1/k} (2^{-i(k-1)}N)^{1/k} = N^{1/k} (\log N)^{1-1/k}.
\end{align*}
By summing this over $i$ and collecting the logarithmic factor from above, we derive the upper bound
\begin{align*}
I_k(1+1/k) \ll N^{1/k} (\log N)^2,
\end{align*}
completing this part of the proof of Theorem \ref{main-Thm}.\\

For most applications of the circle method this kind of estimate is sufficient. H\"older's inequality 
gives us the upper bounds in Theorem \ref{main-Thm}
with an additional factor of size at most $(\log N)^2$. In some applications, however,
it is crucial to have a tight bound, which is only away by a multiplicative constant.
Sections \ref{Pbig} and \ref{Psmall} refine the above argument to obtain this improvement.

\section{Upper Bounds for $p > 1 + 1/k$} \label{Pbig}
The first observation to make is that we can restrict our attention
to the range $p \leq 2$ since the desired result follows for 
$p \geq 2$ from the trivial estimate $|S_k(\alpha)| \leq N$.

Our task is to avoid the logarithmic factors from the estimates in Section \ref{Crit}.
The first logarithm can be dealt with by introducing a polynomial weight before applying
H\"older's inequality. Modify the first calculation in Section \ref{Crit} as follows.
Since $p' \geq 2$, we obtain
\begin{align*}
I_k(p) = & \int_0^1 \Big|\sum_{i \leq H} i^{-1} i T_i(\alpha) \Big|^p \, d \alpha
\ll_p \int_0^1 \Big(\sum_{i \leq H} i^{-p'} \Big)^{p/p'} \sum_{i \leq H} i^p|T_i(\alpha)|^p \, d \alpha\\
 \ll & \sum_{i \leq H} i^p \int_0^1|T_i(\alpha)|^p \, d \alpha.
\end{align*}

To avoid the logarithmic factor in the $L^1$-mean, we interpolate instead between $L^{1+\delta}$ and $L^2$,
for some small $\delta > 0$.
Write $p = (1-\theta) \cdot (1 + \delta) + \theta \cdot 2$, which gives $\theta = \frac{p-(1+\delta)}{1-\delta}$.
By Lemma \ref{L1L2-est} and Lemma \ref{H-int}, we obtain
\begin{align*}
\int_0^1|T_i(\alpha)|^p \, d \alpha \ll_{\delta}  (2^i N^{\delta})^{1-\theta} (2^{-i(k-1)}N)^{\theta}
= N^{\delta (1-\theta) + \theta} \cdot 2^{i(1-\theta) - i(k-1)\theta} = N^{p-1} 2^{-i\phi}
\end{align*}
with $\phi = \frac{1}{1-\delta}((p-1)k-1 - \delta(k-1))$.
The expression $(p-1)k - 1$ is positive, due to the condition $p > 1 + 1/k$ and, therefore, 
when $\delta > 0$ is small enough (dependent on $p$ and $k$), we have $\phi > 0$.
We end up with
\begin{align*}
I_k(p) \ll_p N^{p-1} \sum_{i \leq H} i^p 2^{-i\phi} \ll_p N^{p-1}.
\end{align*}
Due to the exponential decay of $2^{-i\phi}$ our previously introduced factor of $i^p$ doesn't
cause any further problems and we are done with this case of the theorem as well.
We now move towards establishing the lower bounds and leave the technically
more demanding part $p < 1+ 1/k$ for a later section.

\section{Lower Bounds} \label{LB}

To deduce the lower bounds in the range $p < 1 + 1/k$ we use Theorem \ref{BR-Thm1}.
Observe that by Theorem \ref{BR-Thm1} the $L^1$-norm is bounded from below by $N^{1/(k+1)}$.
Applying H\"older's inequality to this with $p' = p/(p-1)$, we obtain 
\begin{align*}
N^{1/(k+1)} \ll \int_0^1 |S_k(\alpha)| \, d\alpha \ll 
\Big(\int_0^1 1^{p'} \, d\alpha \Big)^{1/p'} \Big(\int_0^1 |S_k(\alpha)|^{p} \, d\alpha \Big)^{1/p},
\end{align*}
which gives the required result for $p > 1$. For $p < 1$ we interpolate with Lemma \ref{H-int} between
$p$ and $1 + (k+1)^{-1}$. Write $1 = \theta p + (1-\theta) (1 + (k+1)^{-1})$
and use H\"older's inequality again to get
\begin{align*}
N^{1/(k+1)} \ll \int_0^1 |S_k(\alpha)| \, d\alpha \ll 
\Big(\int_0^1 |S_k(\alpha)|^{p} \, d\alpha \Big)^{1/\theta} 
\Big(\int_0^1 |S_k(\alpha)|^{(k+2)/(k+1)} \, d\alpha \Big)^{1-\theta}.
\end{align*}
The lower bound follows from the upper bound on 
the $L^{(k+2)/(k+1)}$-norm (which is proven in the next section).
For $p > 1 + 1/k$ we immediatly get the correct bound
\begin{align*}
N^{p-1} \ll \int_{|\alpha| \leq 1/(100N)} |S_k(\alpha)|^p \, d\alpha \ll  \int_0^1 |S_k(\alpha)|^p \, d\alpha
\end{align*}
from the $r=1$ case of the following lemma.

\begin{lemma} \label{Major}
Let $a,r \in \N$ with $(a;r^k) = 1$, $r$ squarefree and $r \leq N^{1/(5k)}$. 
Consider $\alpha = a/r^k + \beta$ with $|\beta| \leq 1/(100N)$. 
Then $|S_k(\alpha)| \geq N/(10r^k)$ for $N \geq N_0(k)$.
\end{lemma}

\begin{proof}
First we deal with the case $\beta = 0$ and use summation by parts afterwards.
Insert the decomposition \eqref{mu-identity} of $\mu_k$ to get
\begin{align*}
S_k(a/r^k) = \sum_{n \leq N}  \sum_{d^k|n} \mu(d) e(an/r^k)
= & \sum_{d \leq N^{1/k}} \mu(d) \sum_{m \leq N/d^k}  e(amd^k/r^k).
\end{align*}
The inner sum is $N/d^k + O(r^k)$ if $r|d$ and $O(r^k)$ otherwise, giving
\begin{align*}
S_k(a/r^k) =  N\sum_{\substack{d \leq N^{1/k} \\ r|d}} \frac{\mu(d)}{d^k} + O(r^kN^{1/k})
= \frac{N}{r^k} \sum_{fr \leq N^{1/k}} \frac{\mu(fr)}{f^k} + O(r^kN^{1/k}).
\end{align*}
This gives $S_k(a/r^k) = C(r)N/r^k + O(r^kN^{1/k})$,
where $C(r) = \sum_{f = 1}^\infty \mu(fr)/f^k$ 
satisfies the bounds $1/3 \leq |C(r)| \leq 2$ since
$C(r) = \pm 1 + \sum_{f = 2}^\infty \mu(fr)/f^k$ for squarefree $r$.
The error term $O(r^kN^{1/k})$ is of smaller order due to the restriction $r \leq N^{1/(5k)}$, as long as $N$ is big enough.

Now we introduce a small pertubation of the form $e(\beta n)$ and use
summation by parts to obtain
\begin{align*}
S_k(\alpha)
& = e(\beta N) \sum_{n \leq N} \mu_k(n) e(a n/r^k) - 
\int_1^N 2\pi i \beta e(\beta t) \sum_{n \leq t} \mu_k(n) e(a n/r^k) \, dt.
\end{align*}
The first term is at least of size $N/(4r^k)$ by the previous calculation. 
The second term can be bounded by a similar procedure as above. We have
\begin{align*}
\Big| \sum_{n \leq t} \mu_k(n) e(a n/r^k)\Big| \leq 2t/r^k + r^kt^{1/k}.
\end{align*}
By integration, the contribution of the integral on the right hand side is bounded above by
$2\pi |\beta| N^2/r^k + |\beta|r^kN^{1 + 1/k} \leq N/(10r^k)$ 
by our choice of $\beta$ and $r$, if $N$ is big enough.
\end{proof}

Now we embark on the proof of the lower bound for the remaining case $p = 1 + 1/k$.
Set $R = N^{1/(5k)}$. For $r, s \leq R$ we would have the 
contradiction $(50N)^{-1} \geq |a/r^k - b/s^k| \geq (rs)^{-k} \geq N^{-1}$
if the two points $a/r^k \neq b/s^k$ would have intersecting neighbourhoods of radius $(100N)^{-1}$.
Therefore, they are disjoint and we can use Lemma \ref{Major} to estimate
\begin{align*}
& \int_{0}^1 |S_k(\alpha)|^{(k+1)/k} \,d\alpha \geq 
\sum_{r \leq R} \mu_2(r) \sum_{(a;r^k) = 1} \int_{|\alpha - a/r^k| \leq (100N)^{-1}} |S_k(\alpha)|^{(k+1)/k} \,d\alpha\\
 \geq & \ \sum_{r \leq R} \mu_2(r) \sum_{(a;r^k) = 1} (50N)^{-1} |N/(10r^k)|^{(k+1)/k}
\gg N^{1/k} \sum_{r \leq R} \mu_2(r) \varphi(r^k)/r^{k+1}\\
= & \ N^{1/k} \sum_{r \leq R} \mu_2(r) \varphi(r)/r^{2} \gg N^{1/k} \log N,
\end{align*}
where $\varphi$ is Euler's totient function and the last inequality
follows by summation by parts from the estimate
\begin{align*}
\sum_{r \leq R} \mu_2(r) \varphi(r)/r \gg R,
\end{align*}
which can be deduced as follows. We have
\begin{align*}
& \sum_{r \leq R} \mu_2(r) \varphi(r)/r = \sum_{r \leq R} \varphi(r)/r - \sum_{r \leq R} (1-\mu_2(r)) \varphi(r)/r\\
\geq & R/\zeta(2) + o(R) - \sum_{r \leq R} (1-\mu_2(r))  \geq (2/\zeta(2)-1)R + o(R) \gg R,
\end{align*}
where we used the identities \eqref{mu-identity} and $\varphi(r)/r = \sum_{d|r} \mu(d)/d$.

\section{Upper Bounds for $p < 1 + 1/k$} \label{Psmall}

In this section we essentially repeat an argument from \cite{BR} in a more general setting.
Since for $p \leq 1$ the bound follows by application of H\"older's inequality using the upper bound
for $p > 1$, we can restrict ourselves to the cases $1 < p < 1 + 1/k$.

Our main concern is the logarithmic factor introduced in the $L^1$-bound for $T_i$.
To remove it, we use a smoothed version of our exponential sums instead. The error introduced in this
way turns out to be small in the $L^2$-norm and can eliminated by a $L^0$-$L^2$-interpolation.

As in \cite{BR}, we consider the Fejer kernel
\begin{align*}
0 \leq F(\alpha) = \sum_{|n| \leq N} \Big( 1-\frac{|n|}{N}\Big)e(\alpha n) = 
\frac{\sin^2 (\pi N \alpha)}{N \sin^2 (\pi \alpha)} \ll \min\Big(N,\frac{1}{N\|\alpha\|^2} \Big).
\end{align*}
More generally for $1 \leq K$ and $1 \leq N$, we employ the kernel
\begin{equation} \label{eq-NK}
\begin{split}  
\sum_{|n| \leq N+K} \min \Big(1,\frac{N + K-|n|}{K}\Big)e(\alpha n)
 = & \frac{\sin(\pi (2N + K)\alpha) \sin (\pi K \alpha)}{K \sin^2 (\pi \alpha)}\\
 \ll & \min\Big(N+K,\frac{1}{\|\alpha\|},\frac{1}{K\|\alpha\|^2} \Big).
\end{split}
\end{equation}
We need this estimate also for non-integer values of $K$ and $N$.
This can be derived by taking $[N] + 1$ and $[K] + 1$ instead of $N$ and $K$. 
We obtain almost the same estimate but introduce an additional error of size $O(1)$. 
(The case $0 < K < 2$ is different but easier and can be dealt with separately.)

This can be used to obtain the final refinement with a congruence condition. 
For $1 \leq K$ and $1 \leq N$, $1 \leq d \leq N + K$ and $1 \leq M$ define
\begin{align*}
E_{N,K,M,d}(\alpha) := \sum_{\substack{|n| \leq N+K\\ n \equiv M (d)}} \min \Big(1,\frac{N + K-|n|}{K}\Big)e(\alpha n).
\end{align*}
Using the substitution $n = dm + M-M_0d$ this sum can be rewritten as
\begin{align*}
e((M-M_0d)\alpha) \sum_{|m+r| \leq (N+K)/d} \min 
\Big(1,\frac{(N + K)/d-|m+r|}{K/d}\Big) e(\alpha dm)
\end{align*}
for some $M_0 \geq 0$, such that $0 \leq M - M_0d < d$ and $r := M/d-M_0 < 1$.
We can remove $r$ by adding a summand of size $O(1)$. Now use estimate \eqref{eq-NK} for non-integer $N$ and $K$
to derive
\begin{align*}
E_{N,K,M,d}(\alpha) & \leq \Bigg| \sum_{|m| \leq (N+K)/d} \min 
\Big(1,\frac{(N + K)/d -|m|}{K/d}\Big) e(\alpha dm)\Bigg| + O(1)\\
& \ll \min\Big(\frac{N+K}{d},\frac{1}{\|d\alpha\|},\frac{d}{K\|d\alpha\|^2} \Big) + O(1).
\end{align*}

Now that we have the main new ingredient for this section, we 
decompose $T_i(\alpha)$ such that for $M = [N/2]$ and $K = 2^{h+ki-1}$ we have
\begin{align*}
& T_i(\alpha) =  \sum_{|n| \leq M + K} \min\Big(1,\frac{M+K-|n|}{K} \Big) b_i(M + n)e((M+n)\alpha)\\
& - \sum_{2M < n \leq 2M + K} \frac{2M+K-|n|}{K} b_i(n)e(\alpha n) 
- \sum_{n \leq K} \frac{K-|n|}{K} b_i(n)e(-\alpha n) + O(1),
\end{align*}
corresponding to formula (2.9) in \cite{BR}.
The contribution of the second and third term can be dealt with by Lemma \ref{BR-Lemma1} since
the sum is over a short interval of length $K$. We obtain by H\"older's inequality 
\begin{align*}
&\quad \int_0^1 \Big|\sum_{2M < n \leq 2M + K} \frac{2M+K-|n|}{K} b_i(n)e(\alpha n) \Big|^p \,d\alpha\\
 & \leq \Big( \int_0^1 \Big|\sum_{2M < n \leq 2M + K} \frac{2M+K-|n|}{K} b_i(n)e(\alpha n) \Big|^2 \,d\alpha \Big)^{p/2}\\
& \leq  \Big( \sum_{2M < n \leq 2M + K} |b_i(n)|^2 \Big)^{p/2} \ll (K 2^{i(1-k)} + N^{1/k} i^3)^{p/2}
 \ll_p (2^{h + i})^{p/2} + N^{p/(2k)} i^3.
\end{align*}

The third sum can be treated in exactly the same manner.
For the main contribution from the first sum we interpolate with Lemma \ref{H-int}.
Therefore, we first need an estimate for the $L^1$-norm. One has
\begin{align*}
& \int_0^1 \Big|\sum_{|n| \leq M + K} \min\Big(1,\frac{M+K-|n|}{K} \Big) b_i(M + n)e((M+n)\alpha)\Big| \,d\alpha \\
\leq  & \sum_{2^{i-1} \leq d < 2^i} \int_0^1 \Big|\sum_{\substack{|n| \leq M + K\\ d^k| M + n}} \min\Big(1,\frac{M+K-|n|}{K} \Big) 
e((M+n)\alpha)\Big| \,d\alpha\\
=  & \sum_{2^{i-1} \leq d < 2^i} \int_0^1 \Big|\sum_{\substack{|n| \leq M + K\\ n \equiv -M (d^k)}} 
\min\Big(1,\frac{M+K-|n|}{K} \Big) e((M+n)\alpha)\Big| \,d\alpha.
\end{align*}
Now we use the previously derived estimate for $E_{M,K,-M,d^k}(\alpha)$ and obtain the upper bound
for the $L^1$-norm
\begin{align*}
& \sum_{2^{i-1} \leq d < 2^i} \int_0^1 
\min\Big(\frac{M+K}{d^k},\frac{1}{\|d^k\alpha\|},\frac{d^k}{K\|d^k\alpha\|^2} \Big) + O(1) \,d\alpha\\
\ll & \sum_{2^{i-1} \leq d < 2^i} \Big(\Big( 1 + \log(M/K) \Big) + O(1) \Big)
\ll 2^i (1 + k(h-i)).
\end{align*}
The $L^2$-norm is easily treated by Parseval's identity and Lemma \ref{BR-Lemma1} by
\begin{align*}
& \int_0^1 \Big|\sum_{|n| \leq M + K} \min\Big(1,\frac{M+K-|n|}{K} \Big) b_i(M + n)e((M+n)\alpha)\Big|^2 \,d\alpha \\
\leq & \sum_{|n| \leq M + K} |b_i(M + n)|^2 \leq 2 \sum_{1 \leq n \leq N + K} |b_i(n)|^2 \ll  N2^{i(1-k)}.
\end{align*}

The final step is to interpolate between those two bounds. But before we do so, we want to point out that 
we can deal directly with the $T^*$-part of the exponential sum by the elementary estimate $|x + y|^p \ll_p |x|^p + |y|^p$
and Lemma \ref{L1L2-est}, which already gives the right order of magnitude after application of H\"older's
inequality.

Write $J_k(p)$ for the $L^p$-mean of the remaining sum $\sum_{i \leq h} T_i$.
We use a similar trick as in Section \ref{Pbig} to avoid the first logarithmic factor.
Since $p' \geq 2$, we obtain
\begin{align*}
J_k(p) & =  \int_0^1 \Big|\sum_{i \leq h} (1+ h-i)^{-1} (1 + h-i) T_i(\alpha) \Big|^p \, d \alpha \\
& \ll_p \int_0^1  
\sum_{i \leq h} (1 + h-i)^{p}|T_i(\alpha)|^p \, d \alpha
 \ll \sum_{i \leq h} (1 + h-i)^{p} \int_0^1|T_i(\alpha)|^p \, d \alpha.
\end{align*}

This is the point where we decompose $T_i$ into the smoothed part plus three error terms and apply
the elementary estimate $|x + y|^p \ll_p |x|^p + |y|^p$ to separate those.
We interpolate the smooth part between the previously obtained $L^1$- and $L^2$-bounds by using
$\theta = p-1$ and for the error terms we use the bounds obtained above.
For $N \geq N_0$, we have
\begin{align*}
J_k(p)  \ll_p & \sum_{i \leq h} (1+h-i)^{p} \Big( 2^i (1 + k(h-i)) \Big)^{2-p} \Big( N2^{i(1-k)} \Big)^{p-1} \\
 & +  \sum_{i \leq h} (1+h-i)^{p} \Big((2^{h+i})^{p/2} + N^{p/(2k)} i^3\Big)
\end{align*}
The polynomial factors with $1+h-i$ affect only the constants in the following estimates.
If $N \geq N_0$, the second sum above is $O(N^{p/(k+1)})$ and we obtain the bound
\begin{align*}
J_k(p) & \ll_p N^{p-1} \sum_{i \leq h} (1 + k(h-i))^2 2^{i(2-p + (p-1)(1-k))} + N^{p/(k+1)}\\
& \ll_p N^{p-1} 2^{h(2-p + (p-1)(1-k))} + N^{p/(k+1)}
\ll N^{p/(k+1)}
\end{align*}
since $2^h \approx N^{1/(k+1)}$. This concludes the last part of the proof.

\section*{Acknowledgments}

The author is grateful to his supervisor Trevor Wooley for constant support and motivation.
The PhD-work of the author is partially supported by the EPSRC.

\end{document}